\documentclass{amsart}

\usepackage{amssymb}
\usepackage{mathrsfs}
\usepackage{stmaryrd}

\usepackage{hyperref}
\usepackage{enumitem}

\usepackage{comment}
\usepackage{tikz}
\usepackage{tikz-cd}
\usepackage[all,cmtip]{xy}
\usepackage{todonotes}

\newtheorem{theorem}[equation]{Theorem}
\newtheorem{lemma}[equation]{Lemma}
\newtheorem{proposition}[equation]{Proposition}
\newtheorem{cor}[equation]{Corollary}
\newtheorem{maintheorem}{Theorem}

\theoremstyle{definition}
\newtheorem{definition}[equation]{Definition}
\newtheorem{remark}[equation]{Remark}
\newtheorem{warning}[equation]{Warning}
\newtheorem{example}[equation]{Example}

\numberwithin{equation}{section}

\DeclareMathOperator\Der{Der}

\DeclareMathOperator{\Hom}{Hom} 

\DeclareMathOperator{\Spec}{Spec} 

\DeclareMathOperator{\Hilb}{Hilb}

\newcommand{\mc}{s}
\newcommand{\al}{t}

\newcommand{\new}{\mathrm{new}}
\newcommand{\old}{\mathrm{old}}
\newcommand{\mfm}{\mathfrak{m}}
\newcommand{\mfJ}{\mathfrak{J}}

\def\pow#1{ \llbracket  #1 \rrbracket }

\newcommand\cH{\mathcal{H}}

\renewcommand\AA{\mathbb{A}}

\newcommand{\KK}{\Bbbk}

\newcommand\ZZ{\mathbb{Z}}

\newcommand{\mfB}{\mathfrak{B}}

\title{Local Equations for Hilbert Schemes of Points}

\author{Nathan Ilten}
\address{Department of Mathematics,
Simon Fraser University,
8888 University Drive,
Burnaby BC V5A 1S6,
Canada}
\email{nilten@sfu.ca}

\author{Francesco Meazzini}
\address{Dipartimento di Matematica Guido Castelnuovo, Sapienza Universit\`a di Roma, Piazzale Aldo Moro 5, 00185 Roma, Italy}
\email{francesco.meazzini@uniroma1.it}

\author{Andrea Petracci}
\address{Dipartimento di Matematica, Universit\`a di Bologna, Piazza di Porta San Donato 5, 40126 Bologna, Italy}
\email{a.petracci@unibo.it}

\begin{document}

\begin{abstract}
We compute the completion of the local ring of the Hilbert scheme of degree $n+1$ subschemes of $\AA^n$ at the point corresponding to the ideal $\langle x_1,\ldots,x_n\rangle^2$, and describe the completion of the universal family. For the purposes of comparison, we do this computation with both classical and DGLA methods. We use our explicit equations to produce high dimensional linear subspaces of the Hilbert scheme, and compare our equations with those coming from deformations of based algebras.
\end{abstract}

\maketitle

\section{Introduction}\label{sec:intro}
We always work over a field of characteristic zero denoted by $\KK$.
For integers $d,n\geq 1$, let $\Hilb_{d}^{n}$ denote the Hilbert scheme of $0$-dimensional degree $d$ closed subschemes of the affine space $\AA^n$.
In the polynomial ring $S = \KK[x_1,\dots,x_n]$, consider the ideal $I = \langle x_1,\dots,x_n\rangle^2 \subset S$. This cuts out a degree $n+1$ zero-dimensional closed subscheme of $\AA^n$, and hence gives a $\KK$-rational point $[I]\in \Hilb_{n+1}^{n}$.

The purpose of this note is twofold. Firstly, we explicitly compute the local structure of $\Hilb_{n+1}^{n}$ at the point $[I]$. Secondly, we do this in two ways: using the classical method of lifting syzygies, but also using the more modern approach via differential graded lie algebras (DGLAs). We hope that this side-by-side comparison of methods helps illuminate the more abstract DGLA method.

In order to state our main result, we introduce a little more notation. We assume throughout the paper that $n\geq 3$ is fixed.
Consider the power series ring
\[
	\KK\pow{\al_{ij}^k}:=\KK \pow{\al_{ij}^k \mid i,j,k \in \{ 1,\dots,n\}, i \leq j  }
\]
in $\frac{n^2(n+1)}{2}$ variables. We set $\al_{ij}^k=\al_{ji}^k$ whenever $i>j$.\footnote{Note that the $k$ in $\al_{ij}^k$ is an index, not an exponent. To avoid any confusion, these variables will \emph{never} appear with exponents. The same convention applies to the $\gamma_{ijk}^\ell$ defined in \eqref{eqn:gamma}.}
For indices $1\leq i,j,k,\ell \leq n$ we define 
\begin{equation}\label{eqn:gamma}
	\gamma_{ijk}^\ell=\sum_{\lambda=1}^n \al_{ij}^{\lambda}\al_{k\lambda}^{\ell}-\al_{ik}^{\lambda}\al_{j\lambda}^{\ell}\in \KK\pow{\al_{ij}^{k}}.
\end{equation}
	 We then let $\mfJ$ be the ideal of $\KK\pow{\al_{ij}^{k}}$
generated by 
\begin{align*}
	&\gamma_{ijk}^{\ell}\qquad\qquad &1\leq i,j,k,\ell\leq n,\ \textrm{with}\ j, k,\ell\ \textrm{distinct};\\
	&\gamma_{ijk}^{k}-\gamma_{ij\ell}^{\ell} & 1\leq i,j,k,\ell\leq n,\ j\neq k,\ j\neq \ell.
\end{align*}

\begin{maintheorem}\label{thm:main}
Let $n\geq 3$.
	The completion of the local ring of $\Hilb_{n+1}^{n}$ at the point $[I]$ is isomorphic to the quotient of the power series ring $\KK\pow{\al_{ij}^{k}}$
	by $\mfJ$. 
The universal family in a formal neighborhood of $[I]$ is given by the ideal of $S\pow{\al_{ij}^k}$ generated by
\begin{align*}
	x_ix_j+\sum_{k=1}^n\left( \al_{ij}^{k}x_k+\frac{\gamma_{ijk}^{k}}{n-1}\right)\qquad 1\leq i \leq j \leq n.
\end{align*}
Furthermore, one obtains the miniversal base space and family of $\Spec S/I$ from the above by setting $\al_{ij}^{k}=0$ for $i=j=k$.
\end{maintheorem}

In addition to the two deformation-theoretic approaches to the local structure of $\Hilb_{n+1}^n$ mentioned above, we also approach the above theorem through Poonen's \emph{based algebras} \cite{poonen}. In fact, in our Theorem \ref{thm:Bn}, we are able to show that the generators of $\mfJ$ describe not just the completion of the local ring of $\Hilb_{n+1}^n$ at $[I]$ but cut out an open affine neighborhood of this point. We also use our explicit equations to construct linear subspaces of $\Hilb_{n+1}^n$ whose dimension is very close to the dimension of the whole Hilbert scheme (Corollary \ref{cor:linear}).

The local structure of $\Hilb_{n+1}^{n}$ at the point $[I]$ is of particular interest since this is the ``worst'' point of any Hilbert scheme of points. Indeed, any degree $n+1$ subscheme of $\AA^n$ may be degenerated to $\Spec S/I$ after change of coordinates, see Lemma \ref{lemma:worst}. On the other hand, the completion of the local ring of $\Hilb_d^{n}$ at any given point agrees up to a smooth factor with that of some point of $\Hilb_{m+1}^{m}$ for some sufficiently large $m$.
Indeed, the natural map of deformation functors from embedded deformations of a subscheme of $\AA^n$ to abstract deformations is smooth (\cite[Ex.~10.1]{hartshorne}) so 
we may re-embed any zero-dimensional degree $d$ scheme $X$ in $\AA^m$ with $d\leq m+1$ without changing the completion of the local ring of the Hilbert scheme at this point (up to smooth factor). Taking the union of $X$ with $m+1-d$ disjoint points gives a point in $\Hilb_{m+1}^m$ whose local ring again agrees up to completion and smooth factor with that of $\Hilb_d^n$ at $[X]$.  

Hilbert schemes of points are incredibly well studied. The scheme $\Hilb_{d}^2$ is smooth, see \cite{fogarty}. $\Hilb_{d}^n$ is irreducible for $d\leq 7$ and reducible for $d=8$ and $n\geq 4$ \cite{erman}. Local equations for $\Hilb_{4}^3$ are given in \cite[Example 18.31]{CCA}.
As far as we know, the equations we write down for $\Hilb_{5}^4$ are new.
See \cite{jelisiejew} for open problems in the area and many other references.

The remainder of this note is organized as follows. In \S \ref{sec:prelim} we consider preliminaries that we will refer to throughout the paper. The classical deformation theory approach is in \S \ref{sec:classical} while the approach via DGLAs is in \S \ref{sec:dgla}. These two sections may be read largely independently of each other, but we encourage the reader to make comparisons. To facilitate such comparisons, the reader may consult Table \ref{table:one} which helps translate between key steps in both approaches. Although we obtain the equations describing the local Hilbert scheme via both approaches, we provide the description of the universal family and full proof of Theorem \ref{thm:main} only via the classical approach (cf. \S\ref{sec:proof}).

In \S \ref{sec:algebra} we take our third approach to Theorem \ref{thm:main}, this time via deformations of based algebras \cite{poonen}. Finally, we discuss linear subspaces of $\Hilb_{n+1}^n$ in \S \ref{sec:linear}.

\begin{table}
	\begin{tabular}{l @{\qquad}c @ {\qquad} c }
	& Classical (\S\ref{sec:classical}) & DGLA (\S\ref{sec:dgla}) \\
\hline
\\
1st order deformations & $f^{(1)}(e_{\ell m})$ \eqref{eqn:t1} & $\phi(e_{\ell m})$ \eqref{phi:1} \\
\\
Lifted relations & $r^{(1)}$ \eqref{eqn:rel2} & $\phi(e_{ij}\wedge e_{ik})$ \eqref{phi:2} \\
\\
Lifted rels. satisfied & Lemma \ref{lemma:relsat} & $\phi$ is closed (Lemma \ref{lemma:relsat2})\\
\\
2nd order obstruction & Lemma \ref{lemma:order2} & Lemmas \ref{lemma:phi} and \ref{lemma:phi0}\\
	\\
	2nd order perturbation & $c_{\ell m}$ (Lemma \ref{lemma:order2}) & $\psi(e_{\ell m})$ (Lemma \ref{lemma:phi0})\\
	\\
\end{tabular}
\caption{Rosetta Stone for translating between classical and DGLA approaches}\label{table:one}
\end{table}

\subsection*{Acknowledgements} The authors thank Joachim Jelisiejew and Paolo Lella for helpful conversations.

NI was partially supported by NSERC.

FM is a member of the GNSAGA and is partially supported by the PRIN 20228JRCYB ``Moduli spaces and special varieties'', of ``Piano Nazionale di Ripresa e Resilienza, Next Generation EU''.

AP acknowledges funding from INdAM-GNSAGA and from the European Union - NextGenerationEU under the National Recovery and Resilience Plan (PNRR) - Mission 4 Education and research - Component 2 From research to business - Investment 1.1, Prin 2022, ``Geometry of algebraic structures: moduli, invariants, deformations'', DD N.~104, 2/2/2022, proposal code 2022BTA242 - CUP J53D23003720006.

\section{Preliminaries}\label{sec:prelim}
\subsection{The worst point of the Hilbert scheme}
The following is well-known to experts; we include short proofs for the convenience of the reader.

\begin{lemma}\label{lemma:initideal}Fix a graded term order on $S$.
Let $J$ be any ideal in $S$ cutting out a zero-dimensional scheme of degree $n+1$.
Then the initial ideal of $J$ is $I$ if and only if the residue classes of $1,x_1,\ldots,x_n$ form a basis of $S/J$.
\end{lemma}
\begin{proof}
Suppose that the initial ideal of $J$ is $I$. Since $1,x_1,\ldots,x_n$ are the standard monomials of $I$, they form a basis for $S/J$.

Conversely, suppose that the residue classes of $x_0:=1,x_1,\ldots,x_n$ form a basis of $S/J$. Then for any element $y\in S$, there exist $c_0,\ldots,c_n\in \KK$ such that 
\[
f_y:=y-\sum c_i x_0\in J.
\]
Taking $y=x_ix_j$ for $1\leq i,j\leq n$, we obtain that the initial form of $f_y$ is $x_ix_j$, and so $I$ is contained in the initial ideal of $J$. Since both $S/I$ and $S/J$ have degree $n+1$, we conclude $I$ is the initial ideal of $J$.
\end{proof}

\begin{lemma}\label{lemma:worst}
Let $J$ be any ideal in $S$ cutting out a zero-dimensional scheme of degree $n+1$. After change of coordinates, there is a Gr\"obner degeneration from $\Spec S/J$ to $\Spec S/I$. 
\end{lemma}
\begin{proof}
First we perform a change of coordinates so that the residue classes of $1,x_1,\ldots,x_n$ form a basis of $S/J$. To see that this is possible,
let $k\geq 0$ be the maximum such that the classes of $1,x_1,\ldots,x_k$ in $S/J$ are linearly independent. If $k\neq n$, let $y\in S$ be such that 
$1,x_1,\ldots,x_k,y$ are linearly independent in $S/J$; this exists since $\dim_\KK S/J=n+1$. Performing the coordinate change $x_{k+1}\mapsto x_{k+1}+y$, we may obtain that the classes of $1,x_1,\ldots,x_{k+1}$ are linearly independent. Thus, by induction we may perform a coordinate change so that the classes of $1,x_1,\ldots,x_n$ in $S/J$ are linearly independent, hence form a basis as desired.

To conclude the proof of the lemma, we now apply Lemma \ref{lemma:initideal} to obtain that after change of coordinates, $I$ is an initial ideal of $J$. This gives a Gr\"obner degeneration from $\Spec S/J$ to $\Spec S/I$.
\end{proof}
\subsection{Relations among $\gamma$}
We record some easy relations among the $\gamma_{ijk}^{\ell}$ (see \eqref{eqn:gamma}).
Firstly, it is immediate from the definition that
\begin{equation}\label{eqn:gamma1}
	\gamma_{ijk}^{\ell}=-\gamma_{ikj}^{\ell}.
\end{equation}
It is similarly straightforward to verify that 
\begin{equation}
	\gamma_{ijk}^{\ell}+\gamma_{jki}^{\ell}+\gamma_{kij}^{\ell}=0.\label{eqn:gamma2}
\end{equation}

From these relations, we obtain:
\begin{lemma}\label{lemma:gens}
	The generators of $\mfJ$ of the form
\[
	\gamma_{ijk}^{k}-\gamma_{ij\ell}^{\ell}
\]
for $1\leq i,j,k,\ell,\leq n$, $j$ distinct from $k,\ell$ may be replaced by
\[
	\gamma_{ijk}^{k}-\gamma_{ji\ell}^{\ell}
\]
with
$1\leq i,j,k,\ell\leq n,\ j\neq k,\ i\neq \ell$.
\end{lemma}
\begin{proof}
For $i=j$ the two types of generators agree. We suppose that $i\neq j$. Then on the one hand, for $j$ distinct from $k,\ell$ we may write
\[
	\gamma_{ijk}^{k}-\gamma_{ij\ell}^{\ell}= (\gamma_{ijk}^{k}-\gamma_{ji\lambda}^{\lambda})-(\gamma_{ij\ell}^{\ell}-\gamma_{ji\lambda}^{\lambda})
\]
for any $\lambda\neq i$.
On the other hand, for $j\neq k$ and $\ell\neq i$ we have
\begin{align*}
	\gamma_{ijk}^{k}-\gamma_{ji\ell}^{\ell}&=\gamma_{ijk}^{k}-\gamma_{ji\ell}^\ell+\gamma_{ji\lambda}^\lambda-\gamma_{ji\lambda}^\lambda\\
	&=(\gamma_{ji\lambda}^\lambda-\gamma_{ji\ell}^\ell)+\gamma_{ijk}^k+\gamma_{i\lambda j}^{\lambda}+\gamma_{\lambda j i}^{\lambda}\\
	&=(\gamma_{ji\lambda}^\lambda-\gamma_{ji\ell}^\ell)+(\gamma_{ijk}^k-\gamma_{ij\lambda}^{\lambda})+\gamma_{\lambda j i}^{\lambda}\\
\end{align*}
for any $\lambda$, using \eqref{eqn:gamma2} for the second equality and \eqref{eqn:gamma1} for the third.
Choosing $\lambda \neq i, j$ we see that this belongs to $\mfJ$.
\end{proof}

\subsection{The Taylor complex}\label{sec:taylor}
Recall that we have fixed an integer $n \geq 3$. Consider
\begin{align*}
	p:= n + {n \choose 2}, \qquad q:= {p \choose 2}.
\end{align*}
Our ideal $I=\langle x_1,\ldots x_n\rangle $ of $S = \KK[x_1,\dots,x_n]$ has exactly $p$ generators, namely $x_i^2$, for $1 \leq i \leq n$, and $x_i x_j$, for $1 \leq i < j \leq n$.

The relations among these generators are given by the Taylor complex, which is a free resolution of $S/I$ by finite free $S$-modules \cite[\S4]{CCA}. Actually we are interested only in the part of cohomological degrees\footnote{Throughout we use cohomological conventions, that is, the differential increases the degree by $1$.} $\geq -2$
\begin{equation} \label{eq:Taylor_resolution}
S^{ q} \overset{r}\longrightarrow
S^{ p} \overset{f}\longrightarrow S \longrightarrow 0
\end{equation}
which we describe below.

The elements of the standard $S$-basis of $S^{ p}$ are denoted by $e_{ij}$, where $1 \leq i \leq j \leq n$.
If $1 \leq j < i \leq n$ then we set $e_{ij} := e_{ji}$.
The $S$-linear homomorphism $f$ maps $e_{ij}$ to $f_{ij} := x_i x_j \in S$ for all $1 \leq i,j \leq n$.
We may identify $S^{ q}$ with $\bigwedge^2 S^{ p}$. Under this identification, the $S$-linear homomorphism $r$ is given by
\begin{equation} \label{eq:r}
	r(e_{ij}\wedge e_{k\ell})=\frac{-x_kx_\ell e_{ij}+x_ix_je_{k\ell}}{\prod_{\lambda\in \{i,j\}\cap\{k,\ell\}} x_\lambda}
\end{equation}
	When $\{i,j\}$ and $\{k,\ell\}$ are disjoint, $r(e_{ij}\wedge e_{k\ell})$ is just the Koszul relation between $e_{ij}$ and $e_{k\ell}$. However, when 
$\{i,j\}$ and $\{k,\ell\}$ intersect, $r(e_{ij}\wedge e_{k\ell})$ is obtained from the Koszul relation by dividing by a common variable.

The ideal $I$ is $\ZZ$-graded (with generators in degree $2$); the Taylor complex we consider above is also $\ZZ$-graded. This grading will be inherited by most subsequent objects we consider.
\subsection{The tangent space}
The tangent space at $[I]$ of $\Hilb_{n+1}^n$ is given by first order embedded deformations of $\Spec S/I$. These are in turn in bijection with elements of 
$\Hom_S(I,S/I)$ (see e.g.~\cite[\S1.2]{hartshorne}). 

\begin{lemma}
For every $1 \leq i,j,k \leq n$ the assignment
\[
\theta^{ij}_{k} \colon x_\ell x_m \mapsto \begin{cases}
x_k + I &\quad \text{if } \{\ell,m\} = \{i,j\} \\
0 + I &\quad \text{if } \{\ell,m\} \neq \{i,j\},
\end{cases}
\]
for all $1 \leq \ell,m \leq n$,
extends to a unique $S$-linear homomorphism $\theta^{ij}_{k} \colon I \to S/I$.
\end{lemma}

\begin{proof}
The uniqueness is clear because the $x_\ell x_m$ generate $I$ as an $S$-module.
To prove the existence it is enough to show that the relations among these generators of $I$ are mapped to zero by $\theta^{ij}_{k}$.
Such relations correspond to elements in the image of $r \colon S^{ q} \to S^{p}$; it suffices to consider elements of the form $r(e_{ab}\wedge e_{cd})$. By definition of $r$, the image of $r(e_{ab}\wedge e_{cd})$ under $\theta^{ij}_k$ would be the residue class of a homogeneous binomial of degree either $3$ or $2$ in the variables $x_1,\dots,x_n$, and such binomials are zero in $S/I$.
\end{proof}

\begin{proposition}\label{prop:theta}
	The set $\{ \theta^{ij}_k \mid 1 \leq i,j,k \leq n, \ i \leq j \}$ is a $\KK$-basis of $\Hom_S(I,S/I)$. In particular, $\Hom_S(I,S/I)$ is concentrated in degree $-1$.
\end{proposition}

\begin{proof}
	This follows from the fact that the residue classes of $1,x_1,\ldots,x_n$ form a $\KK$-basis of $S/I$. Indeed, this implies that the $\theta^{ij}_k$'s are $\KK$-linearly independent. Moreover, consider an arbitrary $\theta \in\Hom_S(I,S/I)$. After reducing modulo the $\theta^{ij}_k$, we may assume that $\theta$ is such that for all $1\leq \ell,m\leq n$,
\begin{equation*}
\theta (x_\ell x_m) = b_{\ell m}+I
\end{equation*}
with $b_{\ell m}\in \KK$.

Since $n \geq 2$ we can choose $i \in \{1,\dots,n \}$ such that $i \neq \ell$.
From the fact that $\theta$ is $S$-linear we obtain
\begin{equation*}
b_{im} x_\ell +I = x_\ell \cdot \theta (x_i x_m) = \theta(x_i x_\ell x_m) = x_i \cdot \theta(x_\ell x_m) = b_{\ell m} x_i +I
\end{equation*}
and conclude that $b_{\ell m} = 0$.
Therefore we have proved $b_{\ell m} = 0$ for all $\ell,m$.
It follows that arbitrary $\theta$ is in the span of the $\theta^{ij}_k$.
\end{proof}

\begin{remark}\label{rem:t1}
	The tangent space $T^1_{S/I}$ of a miniversal deformation of $\Spec S/I$ is given by isomorphism classes of first order deformations of $\Spec S/I$, which are in bijection with the quotient of $\Hom_S(I,S/I)$ by the image of the natural map
\[
\Der_\KK(S,S/I)\to \Hom_S(I,S/I),
\]
cf.~\cite[Corollary 5.2]{hartshorne}.

By looking at its action on the generators of $I$, we see that in our setting, the image of $\partial/\partial x_i$ in $\Hom_S(I,S/I)$ is exactly
$\sum_{j=1}^n \theta^{ij}_j.$ It follows that a basis of $T^1_{S/I}$ is given by the classes of $\theta^{ij}_k$ with $1\leq i,j,k\leq n$, $i\leq j$, and not all $i,j,k$ equal. As is the case for $\Hom_S(I,S/I)$, $T^1_{S/I}$ is concentrated in degree $-1$.
Its dimension is exactly $(n+2)n(n-1)/2$.
\end{remark}
\subsection{Obstructions}
Let $R$ denote the image of $r$ in $S^p$, with $R_0$ the submodule generated by Koszul relations.

\begin{lemma}\label{lemma:obstruction} The $S$-module $\Hom_S(R/R_0,S/I)$ has no homogeneous elements of degree less than $-2$.
\end{lemma}
\begin{proof}
	The module $R$ has homogeneous generators of degrees $4$ (the Koszul relations) and degree $3$ ($r(e_{ij}\wedge e_{ik})$ for $j\neq k$). Hence, $R/R_0$ has generators of degree $3$, and thus any homogeneous element $\phi$ of $\Hom_S(R/R_0,S/I)$ has degree at least $-3$. However, for $j,k,\ell$ distinct there is a linear syzygy
\[	x_\ell\cdot r(e_{ij}\wedge e_{ik})
-x_k\cdot r(e_{ij}\wedge e_{i\ell})
+x_j\cdot r(e_{ik}\wedge e_{i\ell}).
\]
Since the residue classes of $x_1,\ldots,x_n$ are linearly independent in $S/I$, this means that if $\phi$ has degree $-3$, $\phi$ must be zero.
\end{proof}

\begin{remark}\label{rem:obstruction}
	The second cotangent cohomology module $T^2_{S/I}$ of $S/I$ arises as a quotient of $\Hom_S(R/R_0,S/I)$, cf.~\cite[Proposition 2.3]{stevens}. Lemma \ref{lemma:obstruction} thus implies that $T^2_{S/I}$ is concentrated in degrees at least $-2$. In fact, one can show that it is concentrated in degree $-2$, but we shall not need this here.
\end{remark}

\section{Classical approach}\label{sec:classical}
In this section, we will follow the classical approach of lifting syzygies to compute the local Hilbert scheme at $[I]\in \Hilb_{n+1}^n$, see e.g.~\cite[\S2]{stevens}.

\subsection{First order deformations and lifted relations}
In the following, we will write
$S\pow{\al_{ij}^{k}}$ for \[S\pow{\al_{ij}^{k}\ |\ 1\leq i,j,k\leq n,\ i\leq j}\] and set \[\mfm=\langle \al_{ij}^{k}\ |\  1\leq i,j,k\leq n,\ i\leq j\rangle\subset S\pow{\al_{ij}^{k}}.\]
As in the introduction, we set $\al_{ij}^{k}=\al_{ji}^{k}$ whenever $i>j$.
Let $f^{(0)}:S\pow{\al_{ij}^{k}}^p\to S\pow{\al_{ij}^{k}}$ and $r^{(0)}:S\pow{\al_{ij}^k}^q\to S\pow{\al_{ij}^k}^p$ be the natural extensions of $f,r$. 
We may use the $\theta_k^{ij}$ to construct a perturbation $f^{(0)}+f^{(1)}$ of $f$.

Define $f^{(1)}: S\pow{\al_{ij}^k}^p\to S\pow{\al_{ij}^k}$ via
\begin{equation}\label{eqn:t1}
	f^{(1)}(e_{\ell m}):=\sum_{\substack{i\leq j\\\lambda}} \al_{ij}^\lambda \theta^{ij}_\lambda(x_\ell x_m)
	=\sum_\lambda t_{\ell m}^{\lambda} x_\lambda
	\in S\pow{\al_{ij}^k}
\end{equation}
for $1\leq \ell\leq m\leq n$.
Likewise, define
$r^{(1)}: S\pow{\al_{ij}^k}^q\to S\pow{\al_{ij}^k}^p$ via
\begin{equation}\label{eqn:rel1}
	r^{(1)}(e_{ij}\wedge e_{k\ell})=-f^{(1)}(e_{k\ell}) e_{ij}+f^{(1)}(e_{ij})e_{k\ell}
\end{equation}
when $\{i,j\}$ and $\{k,\ell\}$ are disjoint, and 
\begin{equation}\label{eqn:rel2}
r^{(1)}(e_{ij}\wedge e_{ik})=
\sum_\lambda \left(\al_{ij}^\lambda e_{k\lambda}
-\al_{ik}^\lambda e_{j\lambda}\right)
\end{equation}
for $j\neq k$.

\begin{lemma}\label{lemma:relsat}
This perturbation satisfies
\[
	(f^{(0)}+f^{(1)})\cdot (r^{(0)}+r^{(1)}) \equiv 0 \mod \mfm^2.
\]
\end{lemma}
\begin{proof}
This is equivalent to showing $f^{(0)}r^{(1)}+f^{(1)}r^{(0)}=0$.
For $e_{ij}\wedge e_{k\ell}$ as in \eqref{eqn:rel1},
we have that 	$(f^{(0)}r^{(1)}+f^{(1)}r^{(0)})(e_{ij}\wedge e_{k\ell})$ is equal to 
\[f^{(0)}(-f^{(1)}(e_{k\ell}) e_{ij}+f^{(1)}(e_{ij})e_{k\ell})+f^{(1)}(-f^{(0)}(e_{k\ell}) e_{ij}+f^{(0)}(e_{ij})e_{k\ell})=0.\]

For $e_{ij}\wedge e_{ik}$ as in \eqref{eqn:rel2}, we have that 
$(f^{(0)}r^{(1)}+f^{(1)}r^{(0)})(e_{ij}\wedge e_{ik})$ is equal to
\begin{align*}
	f^{(0)}\left(\sum_\lambda \al_{ij}^\lambda e_{k\lambda}
	-\al_{ik}^\lambda e_{j\lambda}\right)+f^{(1)}(-x_ke_{ij}+x_je_{ik})\\
=	\left(\sum_\lambda \al_{ij}^\lambda x_kx_\lambda
-\al_{ik}^\lambda x_jx_\lambda\right)-x_k\left(\sum_\lambda \al_{ij}^\lambda  x_\lambda\right)+x_j\left(\sum_\lambda \al_{ik}^\lambda x_\lambda\right)
\end{align*}
which clearly vanishes.
\end{proof}
\begin{remark}\label{rem:koszul}
	For any perturbation $\tilde f$ of $f$, relations coming from $e_{ij}\wedge e_{k\ell}$ with $\{i,j\}, \{k,\ell\}$ disjoint are Koszul relations and may be lifted trivially similarly to \eqref{eqn:rel1}. In the following, we will disregard these relations and only consider the non-Koszul relations coming from $e_{ij}\wedge e_{ik}$.
\end{remark}

\subsection{Second order}
To compute the local Hilbert scheme, we need to lift the above perturbations of $f$ and $r$ to higher order. Since the ideal $I$ is $\ZZ$-graded, everything we are considering (e.g.~$\Hom_S(I,S/I)$, $S\pow{\al_{ij}^k}/\mfm^\lambda$) inherits a $\ZZ$-grading, and we may insist that perturbations of $f$ and $r$ remain homogeneous. The degrees of the $\theta^{ij}_k$ are all $-1$, and the corresponding deformation parameters $\al_{ij}^k$ have degree $1$.

To preserve homogeneity, since the generators of $I$ are quadratic, any homogeneous perturbation of $f$ will have degrees at most $2$ in the $\al_{ij}^k$. Hence, a homogeneous lift of $f^{(0)}+f^{(1)}$ to arbitrary order must have the form
$f^{(0)}+f^{(1)}+f^{(2)}$ where
$f^{(2)}: S\pow{\al_{ij}^k}^p\to S\pow{\al_{ij}^k}$ sends $e_{\ell m}$ to $c_{\ell m}\in \KK[\al_{ij}^k]$, a quadratic form. Since $e_{\ell m}=e_{m\ell}$ we also have $c_{\ell m}=c_{m\ell}$.

\begin{lemma}\label{lemma:order2}
	There exist quadratic forms $c_{\ell m}$ in $\KK\pow{\al_{ij}^k}$ such that  the syzygies of $f^{(0)}+f^{(1)}$ modulo $\mfm^2$ lift to syzygies of $f^{(0)}+f^{(1)}+f^{(2)}$ modulo $\mfm^3$ if and only if we impose the additional conditions
	obtained by setting elements of $\mfJ$ equal to $0$.
Moreover, modulo these constraints, the $c_{\ell m}$ are uniquely determined and we have
\[
	c_{\ell m}=\gamma_{\ell m \lambda}^{\lambda}
\]
for any $\lambda\neq m$.
\end{lemma}
\begin{proof}
As noted in Remark \ref{rem:koszul}, we may concentrate on lifting the relations coming from $e_{ij}\wedge e_{ik}$ with $j\neq k$. For degree reasons, it is not possible to homogeneously perturb $(r^{(0)}+r^{(1)})(e_{ij}\wedge e_{ik})$ to higher order. Thus, the condition that the syzygies of $f$ lift to the syzygies of $f^{(0)}+f^{(1)}+f^{(2)}$ modulo $\mfm^3$ becomes
\[
	(f^{(1)}r^{(1)}+f^{(2)}r^{(0)})(e_{ij}\wedge e_{ik})=0
\]
for all $j\neq k$. Expanding this out gives the condition
\[
	\sum_{\lambda,\ell}\left( \al_{ij}^\lambda \al_{k\lambda}^{\ell}x_{\ell}
	-\al_{ik}^\lambda \al_{j\lambda}^{\ell}x_{\ell}\right)+(-x_kc_{ij}+x_jc_{ik})=0
\]
which we may rewrite as 
\[
	\sum_{\ell} \gamma_{ijk}^{\ell}x_\ell=x_kc_{ij}-x_jc_{ik}.
\]

We now compare coefficients in front of the variables $x_\ell$. For $\ell\neq j,k$, we obtain 
$\gamma_{ijk}^{\ell}=0$. For $\ell=j$ or $\ell=k$ we respectively obtain
\[
	c_{ij}=\gamma_{ijk}^{k}\qquad c_{ik}=-\gamma_{ijk}^{j}=\gamma_{ikj}^{j}
\]
with the last equality coming from \eqref{eqn:gamma1}. This shows that we may lift syzygies if and only if 
$\gamma_{ijk}^{\ell}=0$ whenever $j,k,\ell$ are distinct, and for all $i,j$ and $k\neq j$,
$\gamma_{ijk}^{k}$ is independent of $k$, and agrees with $\gamma_{ji\ell}^{\ell}$ whenever $i\neq \ell$. These are exactly the conditions imposed by $\mfJ$, see Lemma \ref{lemma:gens}. It follows that $c_{\ell m}$ also must have the desired form.
\end{proof}
\begin{remark}\label{rem:cij}
It follows immediately from the description of $\mfJ$ that for any $\lambda\neq m$,
\[
	\gamma_{ij\lambda}^\lambda\equiv \sum_{k=1}^n\frac{\gamma_{ijk}^k}{n-1} \mod \mfJ.
\]
\end{remark}
\subsection{Higher order and proof of Theorem \ref{thm:main}}\label{sec:proof}
We may keep trying to lift $f$ to a higher order perturbation $\hat{f}=f^{(0)}+f^{(1)}+f^{(2)}+\ldots$ subject to the condition that the relations $r$ lift to $\hat{r}=r^{(0)}+r^{(1)}+\ldots$ satisfying $\hat{f}\hat{r}\equiv 0$ modulo $\mfm^k+\mfJ$ for $k=2,3,\ldots$. For $k=2$ we have already done this in Lemma \ref{lemma:order2}. For larger $k$, the obstruction to such a lifting lies in 
\begin{equation}\label{eqn:ospace}
	T^2_{S/I}\otimes (\mfm^{k-1}+\mfJ)/(\mfm^k+\mfJ),
\end{equation}
 see e.g.~\cite[Proposition 2.3]{stevens}. 

 We may always choose such liftings to respect the $\ZZ$-grading, in which case the obstruction to lifting modulo $\mfm^k+\mfJ$ lies in the degree $0$ part of \eqref{eqn:ospace}. But since the generators of $\mfm$ all have degree $1$ and $(T^2_{S/I})_u=0$ for $u<-2$ (Remark \ref{rem:obstruction}), the degree $0$ part of \eqref{eqn:ospace} vanishes for $k>2$. This implies that only obstructions to lifting are already in $\mfJ$.
 This shows that the completion of the local ring of $\Hilb_{n+1}^n$ at $[I]$ is indeed $\KK\pow{\al_{ij}^k}/\mfJ$.

 Assuming that we have chosen our liftings $\hat{f},\hat{r}$ of $f^{(0)}+f^{(1)}+f^{(2)}$ and $r^{(0)}+r^{(1)}$ to respect the $\ZZ$-grading, we must have 
 \[
	 \hat{f}=f^{(0)}+f^{(1)}+f^{(2)}
 \]
 since the generators $e_{ij}$ of $S^p$ all have degree $2$. The term $f^{(2)}$ must be as in Lemma \ref{lemma:order2}, and by Remark \ref{rem:cij} this implies that the universal family is as claimed in Theorem \ref{thm:main}.

 Finally, the same statements about the miniversal deformation of $\Spec S/I$ follow from the same arguments as above where we treat $\al_{ij}^k=0$ as $0$ for $i=j=k$. Indeed, after setting these deformation parameters to $0$, the remaining $\al_{ij}^k$ give a dual basis for $T^1_{S/I}$, see Remark \ref{rem:t1}.
 This concludes the proof of Theorem \ref{thm:main}.

 We conclude this section with the observation of a non-trivial syzygy among the $\gamma_{ijk}^\ell$.
 \begin{lemma}\label{lemma:SYZ}
For any $1\leq i,j,k\leq n$ with $j\neq k$ we have
\[	 \sum_{\ell,\lambda} (\al_{ij}^\ell \gamma_{k\ell\lambda}^{\lambda}
	-\al_{ik}^\ell \gamma_{j\ell\lambda}^{\lambda})\in \mfJ.
\]
\end{lemma}
\begin{proof}
	By the above, there exist homogeneous liftings $\hat{f},\hat{r}$ of $f,r$ satisfying $\hat{f}\hat{r}\equiv 0$ modulo $\mfJ$. We have already seen that $\hat{f}=f^{(0)}+f^{(1)}+f^{(2)}$. For similar degree reasons, for any $i,j,k$ with $j\neq k$ we must have 
	\[
		\hat{r}(e_{ij}\wedge e_{ik})=(r^{(0)}+r^{(1)})(e_{ij}\wedge e_{ik}).
	\]
Indeed, the degree of $r(e_{ij}\wedge e_{ik})$ is $3$ and $S^p$ has generators of degree $2$, so we cannot have any terms involving powers of $\al_{\ell m}^\lambda$ higher than $1$. 

The equality $\hat{f}\hat{r}\equiv 0$ modulo $\mfJ$ thus implies that 
\[
	f^{(2)}r^{(1)}(e_{ij}\wedge e_{ik})\in \mfJ.
\]
But we compute that 
\begin{align*}
	(n-1)\cdot f^{(2)}r^{(1)}(e_{ij}\wedge e_{ik})
	&=(n-1)\cdot \sum_\ell \left(\al_{ij}^\ell f^{(2)}(e_{k\ell})-\al_{ik}^\ell f^{(2)}(e_{j\ell})\right)\\
	&=\sum_{\ell,\lambda}\left(\al_{ij}^\ell \gamma_{k\ell\lambda}^\lambda-\al_{ik}^\ell\gamma_{j\ell\lambda}^\lambda\right)
\end{align*}
with the second equality following from Lemma \ref{lemma:order2} and Remark \ref{rem:cij}. The claim of the lemma follows.
\end{proof}

\begin{remark}
	If we had a direct proof of Lemma \ref{lemma:SYZ}, we would not have to appeal to the argument using degrees of the $\al_{ij}^k$ and of the elements of $T^2_{S/I}$ to conclude that $\mfJ$ contains all obstructions to deforming $\Spec S/I$. 
\end{remark}

\section{DG-Lie methods and abstract deformations}\label{sec:dgla}
We follow the general framework of~\cite{MaMe19} to view deformations of $\Spec S/I$ in terms of those controlled by a differential graded Lie algebra.
\subsection{The derivation Lie algebra}\label{sec:der}
Let $P\to S/I$ be a Koszul-Tate resolution of $S/I$, that is, a semifree\footnote{This means that $P$ is free as a graded commutative algebra after forgetting the differential.} commutative differential graded algebra (CDGA) over $\KK$ with $P^i=0$ for $i>0$, $H^0(P)\cong S/I$ and all other cohomology groups vanishing, see e.g.~\cite{tate}. 
We denote the differential of $P$ by $\delta_P$.\footnote{Note that \cite{tate} uses homological degrees instead of cohomological ones as we do.}
Each $P^i$ splits as a direct sum $P^i=P^i_\old \oplus P^i_\new$, where $P^i_\old$ is generated by all products of elements of higher cohomological degree. For $i<0$, we call generators for $P^i_\new$ as a $P^0$-module \emph{free variables}. In our setting, we can (and will) take $P$ so that $P^0=S$.

The functor of infinitesimal deformations of $P$ in the category of non-positively\footnote{Recall that the restriction to the category of \emph{non-positively graded} CDGAs is crucial, cf.~\cite[Section~7]{MaMe19}.} graded CDGAs is isomorphic to the functor of infinitesimal deformations of $\Spec S/I$ and since $P$ is semifree, these deformations of $P$ all come from perturbing the differential $\delta_P$, see~\cite[Section~7]{MaMe19}.

For $m\in\ZZ$, let $\Hom_\KK^m(P,P)$ consist of those $\KK$-linear maps from $P$ to itself satisfying $\phi(P^i)\subseteq P^{i+m}$ for all $i\in \ZZ$. We take $\Hom^\bullet_\KK(P,P)\subset \Hom_\KK(P,P)$ to be the direct sum of the $\Hom_\KK^m(P,P)$.
\begin{definition}\cite[Example 6.5.6]{manetti}
The derivation Lie algebra of $P$ is 
\begin{align*}
\Der_\KK^\bullet(P,P)&=\bigoplus_m \Der_\KK^m(P,P)\subset \Hom^\bullet_\KK(P,P)
\end{align*}
where $\Der_\KK^m(P,P)$ consists of those $\phi\in\Hom_\KK^m(P,P)$ satisfying the graded Leibniz rule
\[\phi(ab)=\phi(a)b+(-1)^{\ell m}a\phi(b)\ \qquad\forall a\in P_\ell,b\in P.\]
The Lie bracket on $\Der_\KK^\bullet(P,P)$ is  given by
\[
	[\phi_1,\phi_2]=\phi_1\phi_2-(-1)^{m_1m_2}\phi_2\phi_1
\]
for homogeneous derivations of degrees $m_1,m_2$, and the differential is 
\begin{align*}
\partial:\Der_\KK^\bullet(P,P)&\to \Der_\KK^\bullet(P,P)\\
	\phi&\mapsto[\partial_P,\phi].
\end{align*}
\end{definition}

The functor of infinitesimal deformations of $P$ (and thus those of $\Spec S/I$) is isomorphic to the deformation functor defined by the DG Lie algebra $\Der_\KK^\bullet(P,P)$ via Maurer-Cartan solutions modulo gauge equivalence, see~\cite[Corollary~7.12]{MaMe19}. Since the projection $P\to S/I$ is a quasiisomorphism, it induces isomorphisms
\begin{equation}\label{eqn:qiso}
H^i(\Der_\KK^\bullet(P,P))\to H^i(\Der_\KK^\bullet(P,S/I)),
\end{equation}
see~\cite[Section~7]{Hin97}.
Here, $\Der_\KK^\bullet(P,S/I)$ is a subcomplex of $\Hom_\KK^\bullet(P,S/I)$ defined analogously to $\Der_\KK^\bullet(P,P)$ with differential $\overline \partial$ given by
\[
\phi\mapsto \phi\circ \partial_P.
\]
\begin{remark}\label{rmk.doublegrading}
	By definition, $P$ carries a cohomological grading. Since $S/I$ itself carries a $\ZZ$-grading, we may choose $P$ so that each $P^i$ is $\ZZ$-graded, the differential $\partial_P$ is of degree $0$ with respect to this grading, and the isomorphism $H^0(P)\cong S/I$ is an isomorphism of $\ZZ$-graded $S$-modules.  We call this grading coming from the grading of $S/I$ the \emph{internal} grading to distinguish it from the cohomological grading.

Since $P_\new$ is not finitely generated as an $S$-module, arbitrary elements of $\Der_\KK^i(P,P)$ do not necessarily decompose as a sum of (internally) homogeneous derivations. However, we may consider the subalgebra
\[\Der_\KK^{\bullet,\bullet}(P,P)=\bigoplus_m \Der_\KK^{\bullet,m}(P,P)\]
where $\Der_\KK^{\bullet,m}(P,P)$ consists of those $\phi\in\Der_\KK^\bullet(P,P)$ that are (internally) homogeneous of degree $m$.
It follows from \eqref{eqn:qiso} that the inclusion $\Der_\KK^{\bullet,\bullet}(P,P)\to \Der_\KK^{\bullet}(P,P)$ is also a quasiisomorphism, and so
the deformation functor controlled by $\Der_\KK^{\bullet,\bullet}(P,P)$ is isomorphic to that of $\Der_\KK^{\bullet}(P,P)$.
\end{remark}

\subsection{A truncated Koszul-Tate resolution}
For our purposes, we will only need the pieces $P^{-2}\to P^{-1}\to P^0=S$ of a Koszul-Tate resolution of $S/I$. Inductively building $P$ as in \cite[Proof of Theorem 1]{tate}, we may take the complex
\[
	\begin{tikzcd}[row sep=tiny]
		P^{-2} \arrow[dd,equal] & P^{-1} \arrow[dd,equal] & P^0 \arrow[dd,equal]\\
	\\
	\bigwedge^2 S^p  \arrow[r] & S^p \arrow [r,"f"] & S\\
\bigoplus \\
	S^q\arrow[uur,"r"]
\end{tikzcd}
\]
where the top row is the Koszul complex, and $r,f$ are as in \S\ref{sec:taylor}. The Koszul complex is naturally a CDGA; we add the $S^q$ factor to $P^{-2}$ by adjoining $q$ new free variables. In the notation of \S\ref{sec:der}, $P^{-2}_\old=\bigwedge^2 S^p$ and $P^{-2}_\new=S^q$.
As desired, we have $H^0(P)=S/I$ and $H^{-1}(P)=0$. By adjoining additional free variables in degrees less than $-2$, we can also obtain vanishing of the lower cohomology groups; we assume that we have done this in a way such that the resulting $P$ is (internally) $\ZZ$-graded.

\begin{warning}
	By construction, we have identified $S^q$ with $\bigwedge^2 S^p$, although the map $r$ is \emph{not} the usual Koszul differential. To differentiate between elements of the two summands $\bigwedge^2 S^p$ and $S^q$ of $P^{-2}$, we will use the notation $e_{ij}\curlywedge e_{k\ell}$ for the former, and $e_{ij}\wedge e_{k\ell}$ for the latter (as we did in \S\ref{sec:taylor}).  
\end{warning}

\subsection{First order deformations}\label{dgla:t1}
Directly from the definitions, we see 
\begin{align*}
\Der^0_\KK(P,S/I)&\cong \Der_\KK(S,S/I)\\
\Der^1_\KK(P,S/I)&\cong \Hom_S(S^p,S/I)\\
\Der^2_\KK(P,S/I)&\cong \Hom_S(S^q,S/I)
\end{align*}
and
$H^1(\Der_\KK^\bullet(P,S/I))$ is the quotient of
\[
	\{\overline \phi\in \Hom_S(S^p,S/I)\ |\ \overline \phi\circ r=0\}=\Hom_S(I,S/I)
\]
by $\Der_\KK(S,S/I)$. This is exactly $T^1_{S/I}$, see Remark \ref{rem:t1}.
In particular, from Proposition \ref{prop:theta} coupled with this remark, the classes of $\theta_k^{ij}$ for $1\leq i,j,k\leq n$ with not all $i,j,k$ equal give a basis for $H^1(\Der_\KK^\bullet(P,S/I))$.

Viewed as an element of $\Der^1_\KK(P,S/I)$
\[
	\theta^{ij}_{k}(e_{\ell m})= \begin{cases}
x_k + I &\quad \text{if } \{\ell,m\} = \{i,j\} \\
0 + I &\quad \text{if } \{\ell,m\} \neq \{i,j\},
\end{cases}
\]
$\theta^{ij}_{k}(e_{\ell m}\wedge e_{\ell' m'})=0$, and $\theta^{ij}_{k}(e_{\ell m}\curlywedge e_{\ell' m'})$ is determined by the graded Leibniz rule.

In the remainder of \S\ref{sec:dgla}, we will adopt the convention that $t_{ii}^{i}=0$ for any $1\leq i \leq n$.
Set \[\overline\phi=\sum_{i,j,k} \al_{ij}^k\theta_k^{ij}\in \Der^1_\KK(P,S/I)\otimes \KK\pow{\al_{ij}^k};\] we wish to lift this to an element $\phi$ in
$\Der^1_\KK(P,P)\otimes \KK\pow{\al_{ij}^k}$.

\begin{lemma}\label{lemma:relsat2}
	There is a closed element $\phi \in \Der^{1,-1}_\KK(P,P)\otimes \KK\pow{\al_{ij}^k}$ lifting $\overline \phi$ that satisfies
\begin{align}
	\phi(e_{\ell m})&=\sum_\lambda \al_{\ell m}^{\lambda}x_\lambda\qquad 1\leq \ell,m \leq n \label{phi:1}\\
	\phi(e_{ij}\wedge e_{ik}) &= \sum\limits_{\lambda}\left(\al_{ij}^\lambda e_{k\lambda} - \al_{ik}^\lambda e_{j\lambda}\right) \qquad 1\leq i,j,k\leq n\ j\neq k \label{phi:2}
\end{align}
\end{lemma}
\begin{proof}
We may construct $\phi$ by specifying what $\phi$ does to the free variables  of each $P_i$; remaining values of $\phi$ are determined by the graded Leibniz rule. Moreover, any $\phi$ satisfying \eqref{phi:1} and \eqref{phi:2} is a lift of $\overline \phi$, since the image of $\phi$ in $\Der^1(P,S/I)$ is determined by 
\[\phi_{|P^{-1}}:P^{-1}\to P^0=S\to S/I.\]

The condition that $\phi$ be  closed is $[\partial_P,\phi]=0$. Suppose that $[\partial_P,\phi]$ vanishes for all elements of $P$ of cohomological degree at least $k$. Then this commutator also vanishes for products of such elements by the graded Leibniz rule. Assume now $k\leq -2$. Then for any free variable $y\in P^{k-1}$, $\partial_P(y)$ is closed, and since $[\partial_P,\phi](\partial_P(y))=0$ by assumption, $\phi(\partial_P(y))$ is exact and there exists $z\in P^k$ such that setting $\phi(y)=z$ forces $[\partial_P,\phi](y)=0$.

To prove the statement of the lemma, by the previous paragraph it thus remains to show that for $\phi$ satisfying \eqref{phi:1} and \eqref{phi:2}, 
\[[\partial_P,\phi](y)=0\]
for $y=e_{k\ell}$ or $y=e_{ij}\wedge e_{ik}$ $(j\neq k)$. The claim for $y=e_{k\ell}$ is automatic (since $P^{1}=0$). For $y=e_{ij}\wedge e_{ik}$ we have
\begin{align*}
	\partial_P(\phi(e_{ij}\wedge e_{ik}))&=
	\partial_P\left( \sum\limits_{\lambda}\left(\al_{ij}^\lambda e_{k\lambda} - \al_{ik}^\lambda e_{j\lambda}\right)\right)\\
	&= \sum\limits_{\lambda}x_\lambda\left(\al_{ij}^\lambda x_{k} - \al_{ik}^\lambda x_{j}\right)\\
	&=-\phi(-x_k e_{ij}+x_j e_{ik})\\
	&=-\phi(\partial_P(e_{ij}\wedge e_{ik})).
\end{align*}
The claim follows.
\end{proof}

\subsection{The cup product}
The Lie bracket on $\Der^\bullet_\KK(P,P)$ descents to a cup product map on cohomology \[H^1(\Der^\bullet_\KK(P,P))\to H^2(\Der^\bullet_\KK(P,P))\cong H^2(\Der^\bullet_\KK(P,S/I)).\] We are interested in computing the cup product of the class of $\phi$ with itself.

The image of $\frac{1}{2}[\phi,\phi]=\phi\circ \phi$ in $\Der^\bullet_\KK(P,S/I)\otimes \KK\pow{\al_{ij}^k}$ is determined by its action on $P^{-2}$ (all other graded pieces of $P$ map to $0$ since $S/I$ is concentrated in cohomological degree $0$). Moreover, by \eqref{phi:1} and the graded Leibniz rule, $\phi\circ \phi$ maps elements of $P^{-2}_\old=\bigwedge^2 S^p$ to elements of $I\otimes \KK\pow{\al_{ij}^k}$, hence their image in $S/I$ is zero. 
It remains to consider images of $e_{ij}\wedge e_{ik}$ for $j\neq k$.

\begin{lemma}\label{lemma:phi}
	For $i\neq j$,	the image of $\frac{1}{2}[\phi,\phi](e_{ij}\wedge e_{ik})$ in $S/I\pow{\al_{ij}^k}$ is
\[
	\sum_\ell \gamma_{ijk}^{\ell}x_\ell.
\]
\end{lemma}
\begin{proof}
	Using \eqref{phi:1} and \eqref{phi:2}
	we compute 
\begin{align*}
	\phi(\phi(e_{ij}\wedge e_{ik}))=\sum_\lambda\left(\al_{ij}^\lambda \phi(e_{k\lambda}) - \al_{ik}^\lambda \phi(e_{j\lambda})\right) 
	&=\sum_{\lambda,\ell} \left(\al_{ij}^\lambda \al_{k\lambda}^{\ell}x_{\ell} - \al_{ik}^\lambda \al_{j\lambda}^{\ell}x_{\ell}\right)\\
	&=\sum_\ell \gamma_{ijk}^{\ell}x_\ell.
\end{align*}
\end{proof}

\begin{lemma}\label{lemma:phi0}
The image of the class of $\phi$ under the natural map 
\[
	\begin{tikzcd}
		H^1(\Der^\bullet_\KK(P,P))\otimes \KK\pow{\al_{ij}^k}/\mfm^3\arrow[r]&  H^2(\Der^\bullet_\KK(P,S/I))\otimes \KK\pow{\al_{ij}^k}/\mfm^3
	\end{tikzcd}
\]
is zero if and only if we impose the additional conditions on $\al_{ij}^k$ given by setting the elements of $\mfJ$ equal to $0$.

\end{lemma}
\begin{proof}
	The image of the class of $\phi$ in \[\Der^2_\KK(P,S/I)\otimes \KK\pow{\al_{ij}^k}/\mfm^3\cong \Hom(S^q,S/I)\otimes \KK\pow{\al_{ij}^k}/\mfm^3\] is described by Lemma \ref{lemma:phi}.
We wish to understand when it vanishes in cohomology; this is equivalent to the existence of a map
\[\psi\in \Der_\KK^1(P,(S/I)\otimes \KK\pow{\al_{ij}^k}/\mfm^3\cong \Hom_S(S^p,S/I)\otimes \KK\pow{\al_{ij}^k}/\mfm^3\]
such that for all $1\leq i,j,k\leq n$, $j\neq k$, 
\[
	\sum_\ell \gamma_{ijk}^{\ell}x_\ell=-x_k\psi(e_{ij})+x_j\phi(e_{ik}).
\]
Using the linear independence of the $x_1,\ldots,x_n$ and comparing coefficients as in the proof of Lemma \ref{lemma:order2}, the claim of the lemma follows.
\end{proof}

\subsection{The formal Kuranishi map}
For the reader's convenience, we recall here some details on the formal Kuranishi map, see e.g.~\cite[Section~13.2]{manetti}. We shall use the notation introduced in Remark~\ref{rmk.doublegrading}.
Let us fix direct sum decompositions 
\begin{equation}\label{eqn:decomp}
	\Der_\KK^{i,\bullet}(P,P)=B^i\oplus C^i\oplus \cH^i\qquad i=1,2
\end{equation}
with $B^i=\partial \left(\Der_\KK^{i-1,\bullet}(P,P)\right)$ and $B^i\oplus \cH^i=\ker\left( \partial\colon \Der_\KK^{i,\bullet}(P,P)\to \Der_\KK^{i+1,\bullet}(P,P)\right)$.
Let $\pi$ denote projection to $\cH^i$, and let $\sigma:B^2\to C^1$ denote the natural isomorphism induced by the above decompositions.
Note that we have isomorphisms of (internally) graded $\KK$-vector spaces
\[
{H}^i(\Der_\KK^{\bullet,\bullet}(P,P))\cong \cH^i,\qquad i=1,2.
\]

Fixing the above decompositions determines the \emph{formal Kuranishi map}
\[
	K:\widehat{\cH}^1\to \widehat{\cH}^2
\]
	where $\widehat{\cH^i}$ denotes the formal completion with respect to the origin. This map is most readily described in terms of $A$-valued points, where $A$ is any local Artinian $\KK$-algebra with maximal ideal $\mfm_A$ and residue field $\KK$. Indeed, the map is given by
	\[\begin{array}{r c l}
			\Hom(\Spec A,\widehat{\cH}^1)\cong \cH^1\otimes \mfm_A&\to& \cH^2\otimes \mfm_A\cong \Hom(\Spec A,\widehat{\cH}^2)\\
	y\ \ \ &\mapsto& \frac{1}{2}\pi([y,y]-[y,\sigma([y,y])]+\ldots)
\end{array}
\]
where the remaining terms of $\ldots$ are expressions involving $\sigma$ and at least three iterated brackets $[\,,]$.

Despite the fact that the formal Kuranishi map $K$ is not canonically defined and depends on the choice of the splitting~\eqref{eqn:decomp}, it is possible to prove \cite[Section~13.2]{manetti} that the zero fiber of $K$ in $\widehat{\cH}^1$ is independent of the choice of splitting, up to natural isomorphism.
For us, this is relevant because the zero fiber of $K$ is a miniversal base space for $\Spec S/I$, see~\cite{Fuk03}, \cite[Proposition~4.10]{BMM22}, and~\cite[Section~13.2]{manetti}.

\subsection{The miniversal base space for $\Spec S/I$}
We now show that the miniversal base space for $\Spec S/I$ is exactly as described in Theorem \ref{thm:main}.
We can and will assume that we have chosen the  decompositions \eqref{eqn:decomp} to be compatible with the internal $\ZZ$-grading of $\Der_\KK^{\bullet,\bullet}(P,P)$.
In that case, $\sigma$ and $\pi$ are homogeneous of internal degree $0$.
By Remark \ref{rem:obstruction}, $T^2_{S/I}$ has internal degrees at least $-2$; moreover, by \cite[Corollary~12.7]{MaMe19} we have $T^2_{S/I}\cong H^2(\Der_\KK^\bullet(P,P))\cong \cH^2$.
On the other hand $T^1_{S/I}\cong  H^1(\Der_\KK^\bullet(P,P))\cong \cH^1$ is concentrated in internal degree $-1$. 
Since the bracket $[\,,]$ is compatible with the internal $\ZZ$-grading, the only part of $K$ that survives is its quadratic term
\[
	\frac{1}{2}\pi([y,y]).
\]

By \S\ref{dgla:t1}, we may treat the $\al_{ij}^k$ as dual coordinates on $H^1(\Der_\KK^\bullet(P,P))$, and the vanishing locus of the formal Kuranishi map is cut out by the conditions on $\al_{ij}^k$ guaranteeing that the image of $\frac{1}{2}[\phi,\phi]$ in $H^2(\Der_\KK^\bullet(P,P))\cong H^2(\Der_\KK^\bullet(P,S/I))$ vanishes. The claim now follows from Lemma \ref{lemma:phi0}.

 \begin{remark}
	 If instead of the miniversal deformation of $\Spec S/I$ we wish to describe the local Hilbert functor, we would need to replace $\Der^\bullet_\KK(P,P)$ by the DGLA $\Der^\bullet_S (P, P )$, defined analogously to $\Der^\bullet_\KK(P,P)$ but with $S$-linear (as opposed to $\KK$-linear) derivations. The computations we carry out in this section translate verbatim to this setting (where we no longer set $\al_{ij}^k=0$ for $i=j=k$) and one recovers the statement of Theorem \ref{thm:main} for $\Hilb_{n+1}^n$. We leave the details to the reader.
 \end{remark}

 \section{Deformations of based algebras}\label{sec:algebra}
 \subsection{Based algebras}\label{sec:based}
 We now take a third perspective on the local structure of $\Hilb_{n+1}^n$ at $[I]$. Following \cite{poonen}, let $\mfB_{n+1}^1$ be the moduli space of commutative, associative, unital $\KK$-algebras of dimension $n+1$ with fixed basis $v_0,\ldots,v_n$ such that $v_0$ is the multiplicative identity.
 This is an affine scheme (cf.~\cite[Proposition 1.1 and Definition 1.4]{poonen}.
 The ring $S/I$ is a commutative, associative $\KK$-algebra with $1$; it also comes with a distinguished $\KK$-basis given by the residue classes of $1,x_1,\ldots,x_n$. Nakayama's lemma guarantees that there is an open neighborhood $U\subseteq \Hilb_{n+1}^n$ of $[I]$ such that for any point $[J]\in U$, the residue classes of $1,x_1,\ldots,x_n$ remain a basis of $S/J$. Hence, we may identify $U$ with $\mfB_{n+1}^1$.

 \begin{remark}
	 It follows from Lemma \ref{lemma:initideal} that $U$ is exactly the big Bialynicki-Birula cell of $\Hilb_{n+1}^n$ with respect to the $\mathbb{G}_m$-action induced by the diagonal action on $\AA^n$, see \cite{joachim} for a discussion of the Bialynicki-Birula decomposition in this setting.
 \end{remark}

 We may describe a $\KK$-point of $\mfB_{n+1}^1$ in terms of its multiplication constants $\mc_{ij}^k$ for $0\leq i,j,k\leq n$:
 \begin{align*}
	 v_i\cdot v_j=\sum_{k=0}^n \mc_{ij}^{k}v_k\qquad 0\leq i,j \leq n.
 \end{align*}
 Similar to \eqref{eqn:gamma}, we define
for $0\leq i,j,k,\ell \leq n$  
\begin{equation*}
	\widetilde{\gamma}_{ijk}^{\ell}=\sum_{\lambda=0}^n \mc_{ij}^\lambda \mc_{k\lambda}^{ \ell}-\mc_{ik}^\lambda \mc_{j\lambda}^{\ell}
\end{equation*}

\begin{proposition}[cf.~{\cite[Proof of Proposition 1.1]{poonen} }]\label{prop:Bn}
	The scheme $\mfB_{n+1}^1\subset \AA^{(n+1)^3}$ is cut out by the vanishing of the ideal $\widetilde{\mfJ}$ generated by
	 \begin{align}
		 & \mc_{ij}^k-\mc_{ji}^k\qquad 0\leq i,j,k\leq n\label{B1}\\
		 & \mc_{0i}^i-1\qquad0\leq i \leq n\label{B2}\\
		 & \mc_{0i}^j\qquad 0\leq i\neq j\leq n\label{B3}\\
		 & \widetilde{\gamma}_{ijk}^\ell\qquad 0\leq i,j,k,\ell \leq n,\ j\neq k.\label{B4}
	 \end{align}
 \end{proposition}
 \begin{proof}The first equation listed corresponds to commutativity of multiplication, and the second and third to 
$v_0$ acting as the multiplicative identity. Assuming the equations for commutativity, for associativity we consider
\begin{align*}
(v_j\cdot v_i) \cdot v_k-v_j\cdot (v_i\cdot v_k)
&=
\left(\sum_{\lambda=0}^n \mc_{ji}^{\lambda}v_\lambda\right)\cdot v_k
-v_j\cdot \sum_{\lambda=0}^n \mc_{ik}^{\lambda}v_\lambda\\
&=
\sum_{\lambda,\ell} 
\mc_{ji}^{\lambda}\mc_{\lambda k}^{\ell}v_\ell
-\mc_{ik}^{\lambda}\mc_{j\lambda}^{\ell}v_\ell\\
&=\sum_{\ell}\widetilde{\gamma}_{ijk}^{\ell}v_\ell.
\end{align*}
Since the $v_\ell$ are linearly independent, the condition for associativity (assuming commutativity) is thus $\widetilde{\gamma}_{ijk}^{\ell}=0$ for all $0\leq i,j,k,\ell\leq n$.

 \end{proof}

 \subsection{Connection to Theorem \ref{thm:main}} 
 We now relate the above to our equations for the Hilbert scheme from Theorem \ref{thm:main}:
 More specifically, Theorem \ref{thm:Bn} and its proof below give a third way of proving Theorem \ref{thm:main} rather directly by using the ideal from \ref{prop:Bn}. This third proof method is almost (but not quite) independent of the deformation-theoretic approaches considered earlier in the paper, see the discussion following the proof.

 \begin{theorem}\label{thm:Bn}
	 The projection map $\AA^{(n+1)^3}\to \AA^{n^2(n+1)/2}$ obtained by setting 
	 \[
		 \al_{ij}^k=\mc_{ij}^k\qquad 1\leq i,j,k\leq n
	 \]
	 induces an isomorphism between $\mfB_{n+1}^1$ and the subscheme of $\AA^{n^2(n+1)/2}$ cut out by
	 \[\mfJ\cap \KK[\al_{ij}^k\ |\ 1\leq i,j,k\leq n,\ i<j].\]
 \end{theorem}
 \begin{proof}
	 Set 
	 \begin{align*}
		 \widetilde{R}:&=\KK[ \mc_{ij}^k\ |\ 0\leq i,j,k\leq n]\\
		 R:&=\KK[ \al_{ij}^k\ |\ 1\leq i,j,k\leq n,\ i<j]
	 \end{align*}
	 and let $\pi:\widetilde{R}\to R$ be the ring homomorphism defined by 
	 \[
		 \pi(\mc_{ij}^k)= \begin{cases}
			 0 & i=0,j\neq k\ \textrm{or}\ j=0,i\neq k\\
			 1 & i=0,j=k\ \textrm{or}\ j=0,i=k\\
			 -\frac{1}{n-1}\sum_{\lambda=1}^n \gamma_{ij\lambda}^\lambda & i,j\neq 0,k=0\\
			t_{ij}^k & \textrm{else}
		 \end{cases}
	 \]
	 where we maintain the convention that $\al_{ij}^k=\al_{ji}^k$.
We note that the inclusion $\iota:R\to \widetilde R$
sending $\al_{ij}^k$ to $\mc_{ij}^k$ is a section of this map.

We will show that $\pi(\widetilde\mfJ)\subseteq \mfJ\cap R$ and $\iota(\mfJ\cap R)\subseteq \widetilde \mfJ$.
The homomorphisms $\pi$ and $\iota$ then descend to give homomorphisms
\[
\begin{tikzcd}
	\widetilde R/\widetilde \mfJ \arrow[r, bend left,"\overline\pi"] & R/(R\cap \mfJ) \arrow[l,bend left, "\overline\iota"]
\end{tikzcd}
\]
with $\overline\iota$ a section of $\overline\pi$. We will then show that $\overline\iota$ is surjective, hence $\overline \iota$ is an isomorphism, and the claim of the theorem follows.

To that end, let $\widetilde{\mfJ}_0\subset \widetilde{\mfJ}$ be the ideal generated by the elements of \eqref{B1}, \eqref{B2}, \eqref{B3}, and 
$\widetilde\gamma_{ijk}^\ell$ for $0\leq i,j,k,\ell \leq n$ and either $i=0$, $j=0$, $k=0$.
It is immediate from the definition of $\pi$ that $\pi(\widetilde{\mfJ}_0)=0$.

	 We next consider the other generators of $\widetilde \mfJ$, that is, $\widetilde\gamma_{ijk}^\ell$ for $0\leq i,j,k,\ell \leq n$, $j\neq k$, and $i,j,k\geq 1$. 
 We first we suppose that $\ell\geq 1$. Then
 \begin{equation}\label{eqn:tildegamma}
\widetilde{\gamma}_{ijk}^{\ell}
=\sum_{\lambda=1}^n(\mc_{ij}^\lambda\mc_{k\lambda}^\ell-\mc_{ik}^\lambda\mc_{j\lambda}^\ell)  +(\mc_{ij}^{0}\mc_{k0}^{\ell}-\mc_{ik}^{0}\mc_{j0}^{\ell}).
\end{equation}
If $j,k,\ell$ are distinct, 
$(\mc_{ij}^{0}\mc_{k0}^{\ell}-\mc_{ik}^{0}\mc_{j0}^{\ell})\in \widetilde{\mfJ}_0$, and we thus obtain that 
$\pi(\widetilde{\gamma}_{ijk}^{\ell})=\gamma_{ijk}^\ell\in\mfJ\cap R$ and similarly $\iota(\gamma_{ijk}^\ell)\in \widetilde{\mfJ}$.

If $k=\ell$, 
\begin{equation}\label{eqn:sij}
	(\mc_{ij}^{0}\mc_{k0}^{\ell}-\mc_{ik}^{0}\mc_{j0}^{\ell})\equiv s_{ij}^0 \mod \widetilde{\mfJ}_0
\end{equation}
and we obtain
\[
	\pi(\widetilde{\gamma}_{ijk}^{k})=\gamma_{ijk}^{k}-\frac{1}{n-1}\sum_{\lambda=1}^n \gamma_{ij\lambda}^\lambda\in \mfJ\cap R.
\]
Furthermore, for $m\geq 1$, $m\neq j$, we have
\[
\widetilde{\gamma}_{ijk}^{k}-\widetilde{\gamma}_{ijm}^{m}\equiv \sum_{\lambda=1}^n(\mc_{ij}^\lambda\mc_{k\lambda}^\ell-\mc_{ik}^\lambda\mc_{j\lambda}^\ell)
-\sum_{\lambda=1}^n(\mc_{ij}^\lambda\mc_{k\lambda}^m-\mc_{ik}^\lambda\mc_{j\lambda}^m) \mod \widetilde{\mfJ}_0
\]
and we conclude $\iota(\gamma_{ijk}^k-\gamma_{ijk}^m)\in \widetilde\mfJ$. By the previous two paragraphs, we thus obtain that $\iota(\mfJ\cap R)\subseteq \widetilde\mfJ$.

For $j=\ell$ we obtain $\pi(\widetilde{\gamma}_{ijk}^{j})\in \mfJ\cap R$ similar to above.
We now instead take $\ell=0$, and obtain
 \begin{align*}
	 (n-1)\cdot \pi(\widetilde\gamma_{ijk}^0)&=(n-1)\pi(\mc_{ij}^0\mc_{k0}^0-\mc_{ik}^0\mc_{j0}^0)+(n-1)\sum_{\lambda=1}^n\left( \al_{ij}^\lambda \pi(\mc_{k\lambda}^0)-
\al_{ik}^\lambda \pi(\mc_{j\lambda}^0)\right)\\
&=\sum_{m,\lambda=1}^n\left( -\al_{ij}^\lambda \gamma_{k\lambda m}^m+
\al_{ik}^\lambda \gamma_{j\lambda m}^m\right).
 \end{align*}
  By Lemma \ref{lemma:SYZ} we conclude that $\pi(\widetilde{\gamma}_{ijk}^0)\in\mfJ\cap R$.
Thus, $\pi(\widetilde\mfJ)\subseteq \mfJ\cap R$.

Finally, we argue that $\overline\iota$ is surjective. To that end, we must show that modulo $\widetilde \mfJ$, any $\mc_{ij}^k$ is in the subring of $\widetilde R$ generated by those $\mc_{ij}^k$ with $1\leq i,j,k\leq n$ and $i<j$.
This is immediate for all $\mc_{ij}^k$ except when $k=0$. In that case, we may use
\eqref{eqn:tildegamma} with $k=\ell$ coupled with \eqref{eqn:sij}.
 \end{proof}

 By the discussion at the beginning of \S\ref{sec:based}, we conclude that an open affine neighborhood of $\Hilb_{n+1}^n$ is cut out by the ideal $\mfJ\cap \KK[\al_{ij}^k]$. We thus obtain not only a description of the completion of the local ring of $\Hilb_{n+1}^n$ at $[I]$, but in fact a description of the local ring. We note the only place deformation-theoretic methods entered into our proof of Theorem \ref{thm:Bn} was via Lemma \ref{lemma:SYZ}. It would be interesting to give a purely elementary proof of Lemma \ref{lemma:SYZ} not relying on the arguments of \S\ref{sec:classical}.

 \section{Linear subspaces of $\Hilb_{n+1}^n$}\label{sec:linear}

We denote by $U$ the open neighborhood  of $[I]$ in $\Hilb_{n+1}^n$ from \S\ref{sec:based}.
By Theorem \ref{thm:Bn}, we may realize $U$ as the affine subscheme of $\AA^{n^2(n+1)/2}$ cut out by the generators of $\mfJ$. Although we have explicit equations for $U$, they are complicated enough that it is difficult to say anything directly about the component structure of the Hilbert scheme at this point. Of course, one component passing through $[I]$ is the smoothing component, whose general point corresponds to $n+1$ distinct (smooth) points in $\AA^n$. As such, this component has dimension $(n+1)n=n^2+n$.

Somewhat surprisingly, 
there are \emph{linear} subspaces of $U$ whose dimension grows much more rapidly than that of the smoothing component.
Let $A,B$ be disjoint subsets of $\{1,\ldots,n\}$ of sizes $a,b$ and 
set 
\[
	L_{A,B}=\{x\in \AA^{n^2(n+1)/2}\ |\ x_{ij}^k=0\ \textrm{unless}\ i,j\in A,\ k\in B\}.
\]
\begin{proposition}\label{prop:linear}
	$L_{A,B}$ is a linear subspace of $\AA^{n^2(n+1)/2}$ of dimension $a(a-1)b/2$. Moreover, $L_{A,B}\subseteq U$.
\end{proposition}
\begin{proof}
	The dimension count is straightforward. The claim that $L_{A,B}\subseteq U$ follows from the fact that the restriction of any $\gamma_{ijk}^\ell$ to $L_{A,B}$ vanishes, since each monomial in $\gamma_{ijk}^\ell$ involves one term that must vanish. For example, for $\al_{ij}^\lambda \al_{k\lambda}^\ell$, the first or second variable must vanish depending on whether $\lambda \in A$ or $\lambda \in B$.
\end{proof}

\begin{example}
	For $n=16$, we may choose $a=11,b=5$ to obtain a linear subspace of dimension $275$. On the other hand, the smoothing component has dimension $272$, so we conclude that $\Hilb_{15}^{14}$ is not irreducible. Of course, we already knew this was true, see the discussion in \S\ref{sec:intro}.
\end{example}

Maximizing the dimension $a(a-1)b/2$ subject to the constraint $a+b=n$ means $a$ is the floor or ceiling of
\[
	a_{\max}=\frac{(n+1)+\sqrt{n^2-n+1}}{3}.
\]
	Notice that for every $n\geq1$ we have
	\[
	\frac{2n}{3} \leq a_{\max} \leq \frac{2n+1}{3} \; .
	\]

\begin{cor}\label{cor:linear}
The open subset $U$ of $\Hilb_{n+1}^n$ contains a linear subspace of dimension at least
\begin{align*}
\frac{2}{27}n^3-\frac{1}{9}n^2\qquad \textrm{if}\ n\equiv 0\ \mod 3\\
\frac{2}{27}n^3-\frac{1}{9}n^2+\frac{1}{27}\qquad \textrm{if}\ n\equiv 1\ \mod 3\\
\frac{2}{27}n^3-\frac{1}{9}n^2-\frac{1}{9}n+\frac{2}{27}\qquad \textrm{if}\ n\equiv 2\ \mod 3\\
\end{align*}
In particular, this gives a lower bound on the dimension of $\Hilb_{n+1}^n$.
The number of such linear subspaces is at least 
${n}\choose{m}$
where $m$ is respectively $2n/3$, $(2n+1)/3$, and $(2n+2)/3$ if $n\equiv 0,1,\ \textrm{or}\ 2\mod 3$.
\end{cor}
\begin{proof}
We apply Proposition \ref{prop:linear} with $a$ equal to the value $m$ from the statement of the corollary.
\end{proof}

Our lower bound on the dimension of $\Hilb_{n+1}^n$ was already known due to work of Poonen, see \cite[Corollary 4.4, Theorem 9.2]{poonen}.
Rather remarkably, the dimension of a maximal-dimensional linear subspace of $U$ is very close to the dimension of $\Hilb_{n+1}^n$, see Corollary 12.1 of op.~cit.

\bibliographystyle{alpha}
\bibliography{Biblio_Hilbert}

\end{document}